\documentclass[12pt]{article}
\title{Cluster Variables and Perfect Matchings of Subgraphs of the $dP_3$ Lattice}
\author{Sicong Zhang}
\date{October 28, 2012}
\usepackage{color}
\usepackage{amsmath}
\usepackage{amssymb}
\usepackage{amsthm}
\usepackage{graphicx}
\usepackage{url}
\theoremstyle{definition}
\newtheorem*{defn}{Definition}
\newtheorem{rk}{Remark}

\theoremstyle{plain}

\newtheorem{prop}{Proposition}
\newtheorem*{thm}{Theorem}

\begin{document}
\maketitle
\begin{abstract}
We give a combinatorial intepretation of cluster variables of a specific cluster algebra under a mutation sequence of period 6, in terms of perfect matchings of subgraphs of the brane tiling dual to the quiver associated with the cluster algebra. 
\end{abstract}
\section{Introduction}
Cluster algebras, introduced by Fomin and Zelevinsky in the past decade \cite{FZ01,FZ02}, arise natually in diverse areas of mathematics such as total positivity, tropical geometry and Lie theory. Previous work, such as \cite{BPW09,M07,MP06,MS08}, gave combinatorial intepretations for the cluster variables as perfect matchings of graphs, under suitable weighting schemes. Of particular interest is the situation where the graphs are directly related to the quiver of the cluster algebra, namely when they are subgraphs of the dual of the quiver. Some motivating examples include Speyer's proof of the Aztec Diamond Theorem \cite{S04} and Jeong's Tree Phenomenon and Superposition Phenomenon \cite{J11}.

In this report we give a specific example of cluster variables whose Laurent expansions can be intepreted as perfect matchings of subgraphs of the brane tiling dual to the original quiver.


\section{Quiver, $ dP_3 $ Lattice and Cluster Variables}

\begin{figure}
\centering
\includegraphics[width=\textwidth]{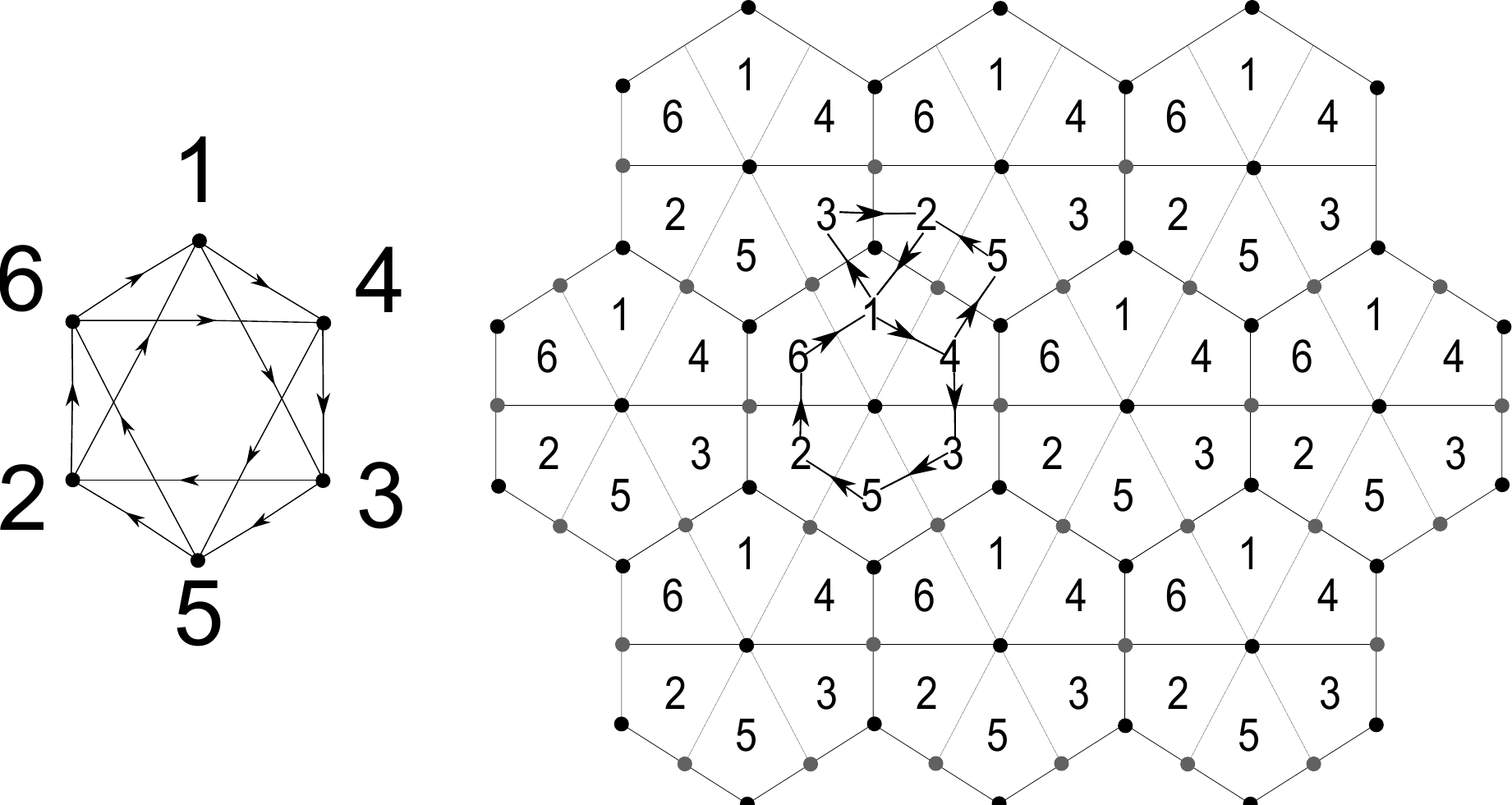}
\caption{Quiver $Q$ and its brane tiling, $dP_3$ lattice}
\label{dP3}
\end{figure}

Figure~\ref{dP3} shows the quiver $Q$ we are working with, which is listed in \cite{HS12} as Model 10 Phase a. The \emph{brane tiling} (doubly periodic planar bipartite graph) dual to $Q$ is a superposition of the triangular lattice with its dual hexagonal lattice; we follow \cite{CY10} and call it the $dP_3$ lattice.

\begin{defn}
Let $\sigma :=(15)(24)(36)$, a permutation of 6 numbers. It corresponds to $180^\circ$ rotation of $Q$ or the lattice.
\end{defn}

\begin{figure}
\centering
\def\svgwidth{\columnwidth}
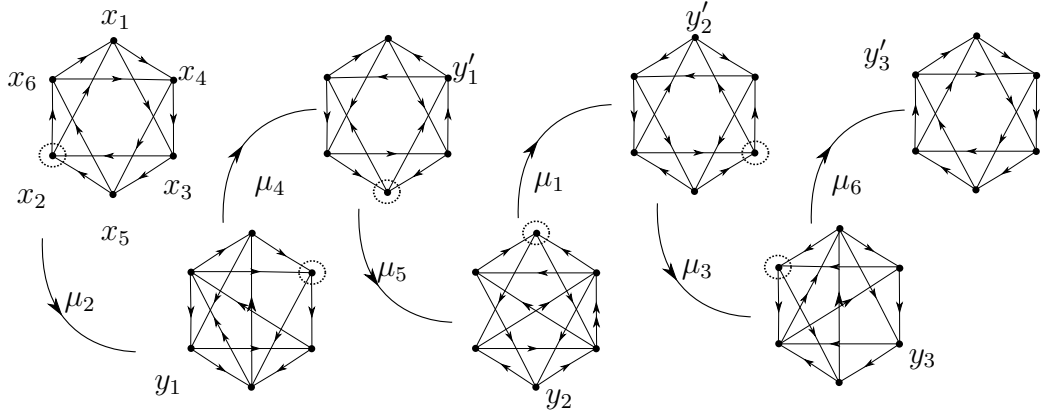
\caption{Quiver $Q$ under the periodic mutation sequence $2, 4, 5, 1, 3, 6,\dotsc$}
\label{Mutation}
\end{figure}

Mutation at node $a$ and at node $\sigma (a)$ commutes since there are no arrows between the two nodes. The two mutations preformed as a pair reverse arrows incident to $a$ and $\sigma (a)$. Therefore after cyclically mutating all 3 pairs of nodes, we will get back $Q$, i.e. $Q$ has period 6. Consider the periodic mutation sequence at vertices $2, 4, 5, 1, 3, 6,\dotsc$ (Figure~\ref{Mutation}). Name the new cluster variables by:
$$\stackrel{\mu_2}{\rightarrow} y_1 \stackrel{\mu_4}{\rightarrow} y'_1 \stackrel{\mu_5}{\rightarrow} y_2  \stackrel{\mu_1}{\rightarrow} y'_2 \stackrel{\mu_3}{\rightarrow} y_3 \stackrel{\mu_6}{\rightarrow} y'_3
\stackrel{\mu_2}{\rightarrow} \dots $$
By symmetry of $Q$ under $\sigma$, $y'_N = \sigma (y_N)$, where $\sigma$ acts naturally on Laurent polynomials via $\sigma (x_i)=x_{\sigma (i)}$. 

For consistency, set $y_{-2}=x_2$, $y'_{-2}=x_4$, $y_{-1}=x_5$, $y'_{-1}=x_1$, $y_{0}=x_3$, $y'_{0}=x_6$. The exchange relation now becomes:
\begin{equation}\label{exchange}
y_N y_{N-3} = y_{N-1}y_{N-2} + y'_{N-1}y'_{N-2}
\end{equation}
for $N \ge 1$.

\section{Diamonds on the $dP_3$ lattice}
We will discuss a sequence of subgraphs on the $dP_3$ lattice, first introduced by Propp \cite{P99}, further studied by Ciucu under the name of ``Aztec Dragons" \cite{C05}, and recently extended to half-integral order by Cottrell and Young \cite{CY10}, in which they are called diamonds.

\begin{defn}
A \emph{strip} is a subgraph of $dP_3$ bounded between two neighboring horizontal lines. A \emph{block} is one of the four following graphs bounding the union of 3 adjacent faces: $[254]$, $[316]$, $[214]$, $[356]$ (Figure~\ref{StripAndBlocks}).
\end{defn}

\begin{figure}
\centering
\includegraphics[width=\textwidth]{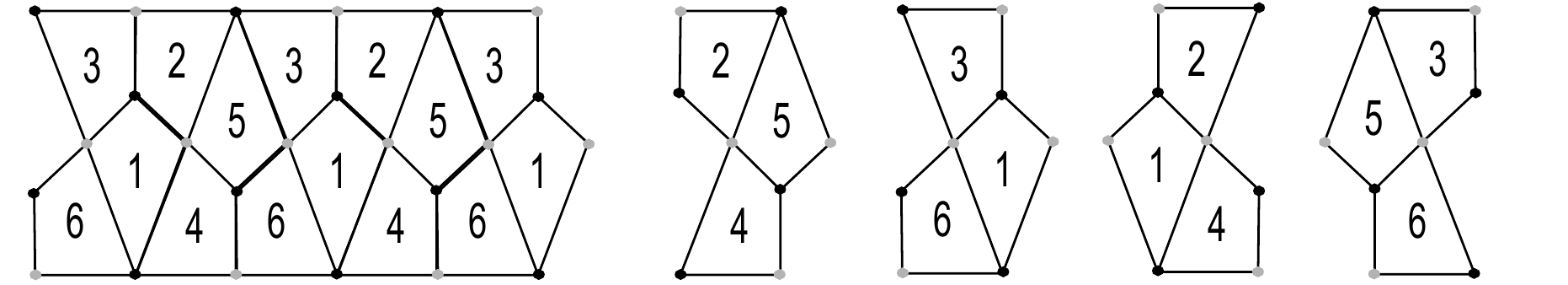}
\caption{A strip and four types of blocks}
\label{StripAndBlocks}
\end{figure}

Fix a strip $s_0$, let $s_j$ be the $j$-th strip north (resp. south) of $s_0$ if $j>0$ (resp. $j<0$). For $j\leq0$, tile $s_j$ alternatingly by blocks $[ 254]$ and $[ 316]$, while for $j>0$, tile $s_j$ by blocks $[ 214]$ and $[ 356]$. Note that each block is adjacent to one block in each of the 4 directions. Label each block $T(i,j)$ such that $T(0,0)$ is $[254]$, block $T(i,j)$ is in strip $s_j$, directly east of block $T(i-1,j)$ and north of $T(i,j-1)$.

\begin{defn}
For $n \in \mathbb{Z}_{>0}$, a \emph{diamond} of order $n$ is
$$D_n=\bigcup_{\left|i+n-1\right|+\left|j\right| \leq n-1} T(i,j)$$
and a \emph{diamond} of order $n+1/2$ is
$$D_{n+1/2}=D_n\bigcup \left(\bigcup_{\substack{-n+2\leq i\leq 1\\ 1\leq j\leq 2-i}} T(i,j)\right) \bigcup S_3 \bigcup S_2$$
where $S_3$ (resp. $S_2$) is the square labeled 3 (resp. 2) in the block $T(1,0)$ (resp. $T(2,0)$).
Also define $D_k=\emptyset $ for $ k \leq 0 $, and $D_{1/2}$ to be a square labeled 2.
$ D'_m $ for $m \in \frac{1}{2} \mathbb{Z}_{\ge 0}$ is obtained by rotating $ D_m $ by $180^\circ$ and relabeling faces according to $ \sigma $.
\end{defn}

\begin{rk}
The definition of integer order diamond here differs from the one given by \cite{CY10} by a reflection about a horizontal line.
\end{rk}

\begin{figure}
\centering
\def\svgwidth{\columnwidth}
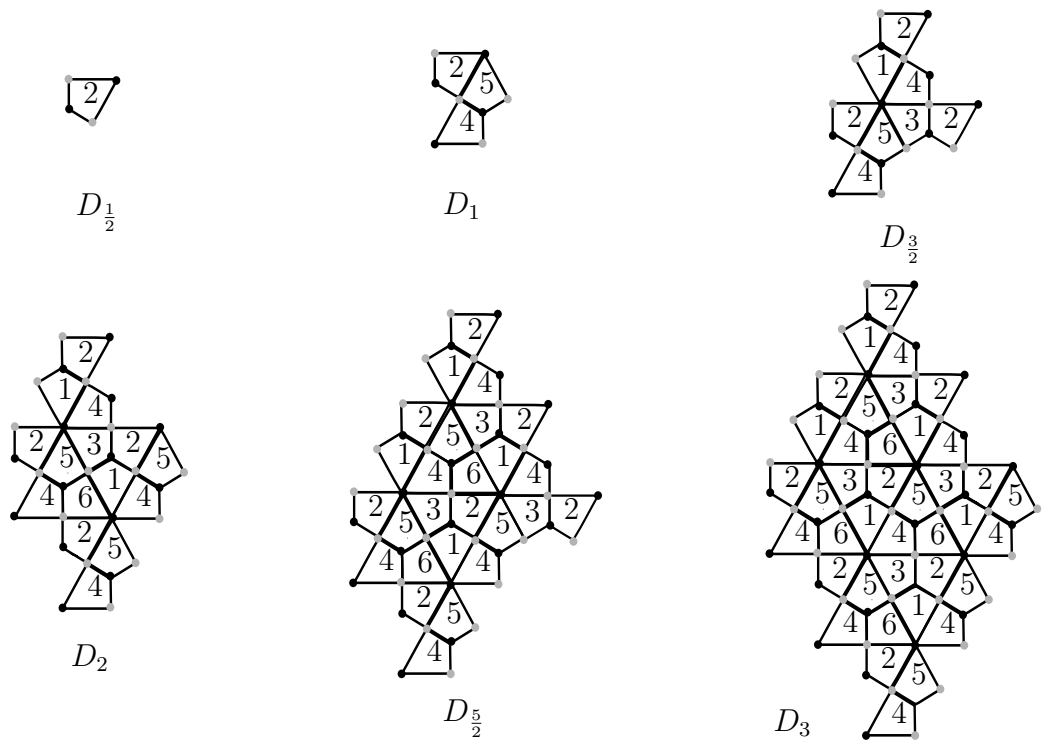
\caption{The first few diamonds, $D_{n/2}$ for $1 \leq n \leq 6$ }
\label{DiamondExample}
\end{figure}

Figure~\ref{DiamondExample} shows the first few diamonds. Note that $D'_{m}$ differs from $D_{m}$ only by 1 or 2 squares.

\begin{figure}
\centering
\def\svgwidth{\columnwidth}
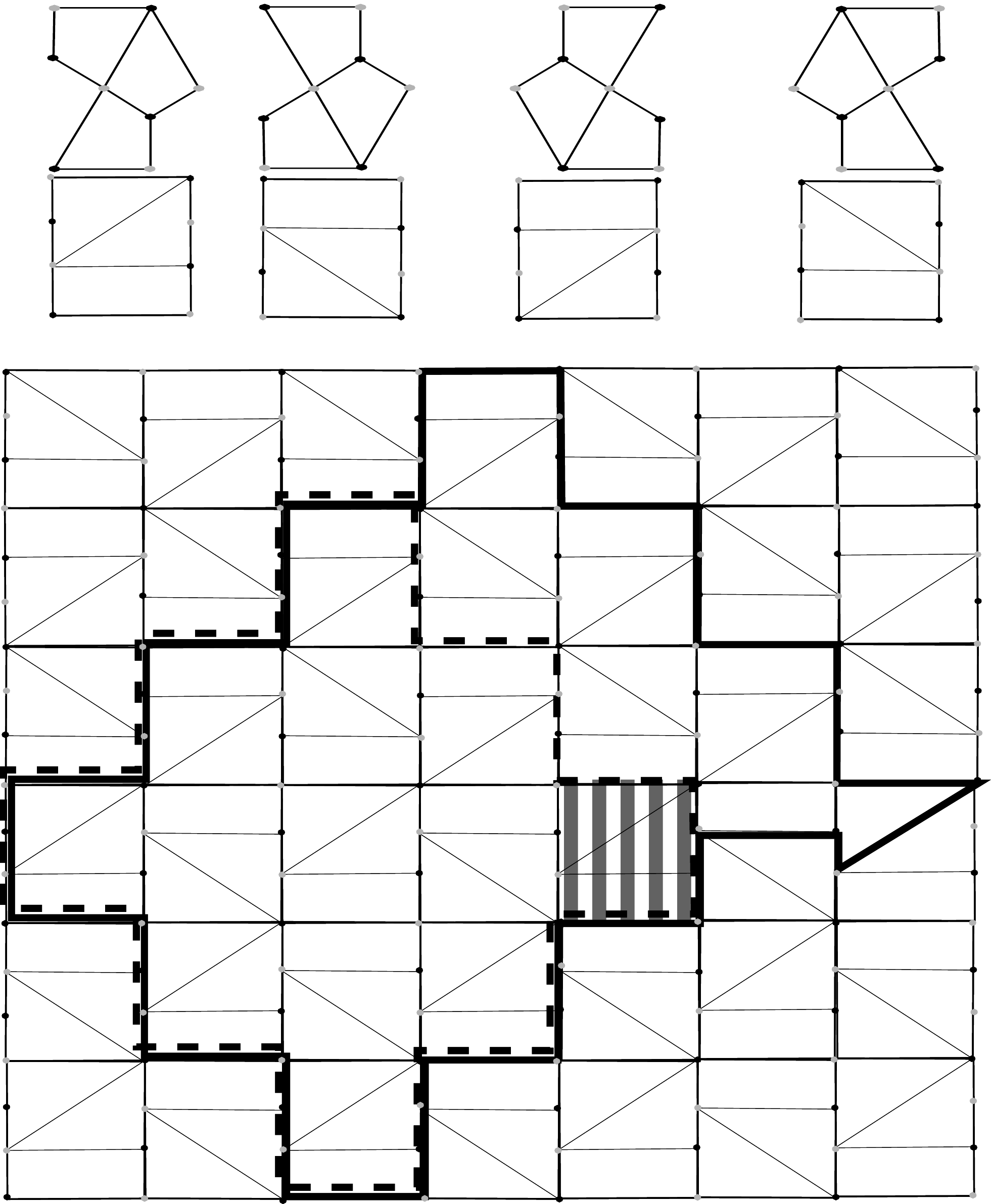
\caption{$D_3$ and $D_{3+1/2}$ drawn in square blocks. The shaded block is $T(0,0)$.}
\label{SquareFormBlocks}
\end{figure}

Figure~\ref{SquareFormBlocks} shows $D_3$ and $D_{3+1/2}$ drawn in square blocks. Informally, for each integer $n$, $D_n$ is the Aztec diamond of order $n$ with squares replaced by blocks, such that the longest row is on $s_0$ and the rightmost block is $\left[ 254\right]$, and $D_{n+1/2}$ resembles an $n$-by-$(n+1)$ Aztec diamond, with two extra squares.

For $m\in \frac{1}{2} \mathbb{Z}_{>0}$, \cite{CY10} gives the number of perfect matchings of $D_m$:
$$ \left|PM(D_m)\right| = \begin{cases} 2^{m(m+1)} &\mbox{if } m \in \mathbb{Z}_{>0} \\ 
                           2^{(m+1/2)^2} & \mbox{if }  m+1/2 \in \mathbb{Z}_{>0}. 
             \end{cases}  $$

\begin{prop}
Setting $ x_i = 1 $, $$ y_N=\left|PM(D_{N/2})\right|.$$
\end{prop}
\begin{rk}
This can be easily checked by induction. The proof is omitted as it will follow from the more general theorem that will be proven later.
\end{rk}

The proposition above makes us suspect that under a suitable weighting on the perfect matchings of $D_{N/2}$, $y_N$ can be expressed as the weight of $D_{N/2}$, up to a monomial factor. By symmetry of $Q$ under $\sigma$, $y'_N$ should be related to $D'_{N/2}$.

We put weights on edges of graphs as follows: for an edge with two neighboring faces labeled $ a $ and $ b $, weight it by $ \frac{1}{x_a x_b} $.

\begin{defn}
The \emph{covering monomial}, $ m(D) $, of a subgraph $ D $ of the brane tiling is a monomial in $\{ {x_i}\}_{1 \leq i \leq 6} $, where the exponent of $ x_i $ counts the number of faces of $ D $ and its neighboring faces in the brane tiling with the label $i$.
\end{defn}

\begin{rk}
There is a more general definition of covering monomial in Jeong's paper \cite{J11}.
\end{rk}

\begin{thm}
For $N \in \mathbb{Z}_{>0}$,
\begin{equation}\label{expansion 1}
y_N = w(D_{N/2})m(D_{N/2}),
\end{equation}
\begin{equation}\label{expansion 2}
y'_N = w(D'_{N/2})m(D'_{N/2}).
\end{equation}
\end{thm}

The theorem follows from comparing the exchange relation with the two recursions on the weights and covering monomials of diamonds.

\section{Recursion on the weights of diamonds}
\begin{prop}
For $ n \in \mathbb{Z}_{\ge 2}$,
 
\begin{align}\label{w recursion 1}
w(D_n)w(D_{n-3/2}) &= w(D_{n-1/2})w(D_{n-1}) \frac{1}{x_1 x_2} \frac{1}{x_3 x_4} \frac{1}{x_5 x_6}\nonumber
\\                 &+ w(D'_{n-1/2})w(D'_{n-1}) \frac{1}{x_1 x_2} \frac{1}{x_2 x_3} \frac{1}{x_3 x_5};
\end{align}

For $ n \in \mathbb{Z}_{\ge 1}$,
\begin{align}\label{w recursion 2}
w(D_{n+1/2})w(D_{n-1}) &= w(D_{n})w(D_{n-1/2}) \frac{1}{x_1 x_3} \frac{1}{x_2 x_6} \frac{1}{x_4 x_5}\nonumber
\\                     &+ w(D'_{n})w(D'_{n-1/2}) \frac{1}{x_1 x_3} \frac{1}{x_2 x_3} \frac{1}{x_2 x_5}.
\end{align}

\end{prop}

\begin{proof}
Let $ m=n$ or $n+1/2$ be a half-integer. We prove this by using graphical condensation, i.e. we need to show that a superposition of perfect matchings of $ D_m $ and $D_{m-3/2} $ can always be decomposed into a matching of $ D_{m-1/2} $ and $D_{m-1} $ or $ D'_{m-1/2} $ and $D'_{m-1} $, with three extra edges, but not in both ways. 

We will use Speyer's \cite[page 37]{S04} formulation of Kuo's graphical condensation theorem \cite{K04} stated as follows. If a planar bipartite graph $G$ has its vertices partitioned into $$V(G)=\mathbf{N} \sqcup \mathbf{NE} \sqcup \mathbf{E} \sqcup \mathbf{SE} \sqcup \mathbf{S}\sqcup \mathbf{SW}\sqcup \mathbf{W}\sqcup \mathbf{NW}\sqcup \mathbf{C}$$ such that:

\begin{enumerate}
	\item Possible edge connection among the nine parts as described in Figure~\ref{EdgeConnection};
     \item \textbf{NW}, \textbf{SE} each contains one more black vertex than white vertex, \textbf{NE}, \textbf{SW} each contains one more white vertex than black vertex, and \textbf{C}, \textbf{N}, \textbf{E}, \textbf{S}, \textbf{W} each contains the same number of vertices in either color;
	\item Boundary vertices are all black for \textbf{NW} and \textbf{SE}, and all white for \textbf{NE}, \textbf{SW}.
\end{enumerate} 

Then 
\begin{align} \label{Condensation}
	w(G)w(\mathbf{C})=&w(\mathbf{W} \cup \mathbf{NW} \cup \mathbf{SW} \cup \mathbf{C})w(\mathbf{E} \cup \mathbf{NE} \cup \mathbf{SE} \cup \mathbf{C})w(\mathbf{S})w(\mathbf{N})+  \nonumber
\\   &w(\mathbf{S} \cup \mathbf{SE} \cup \mathbf{SW} \cup \mathbf{C})w(\mathbf{N} \cup \mathbf{NE} \cup \mathbf{NW} \cup \mathbf{C})w(\mathbf{W})w(\mathbf{E}).
\end{align}
\begin{figure}
\centering
\includegraphics[scale=0.5]{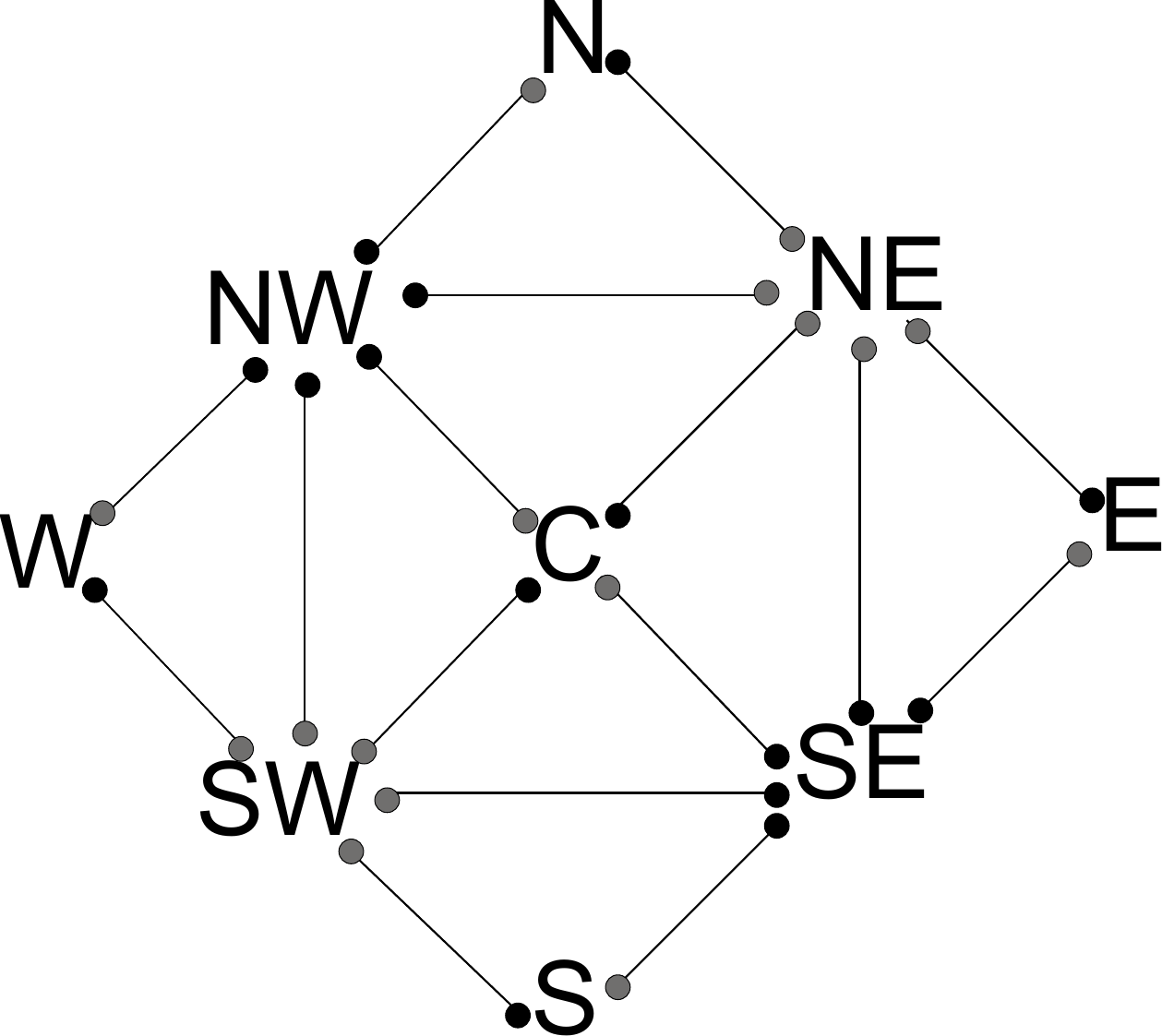}
\caption{Possible edge connection among nine vertex components}
\label{EdgeConnection}
\end{figure}

\begin{figure}
\centering
\includegraphics[width=\textwidth]{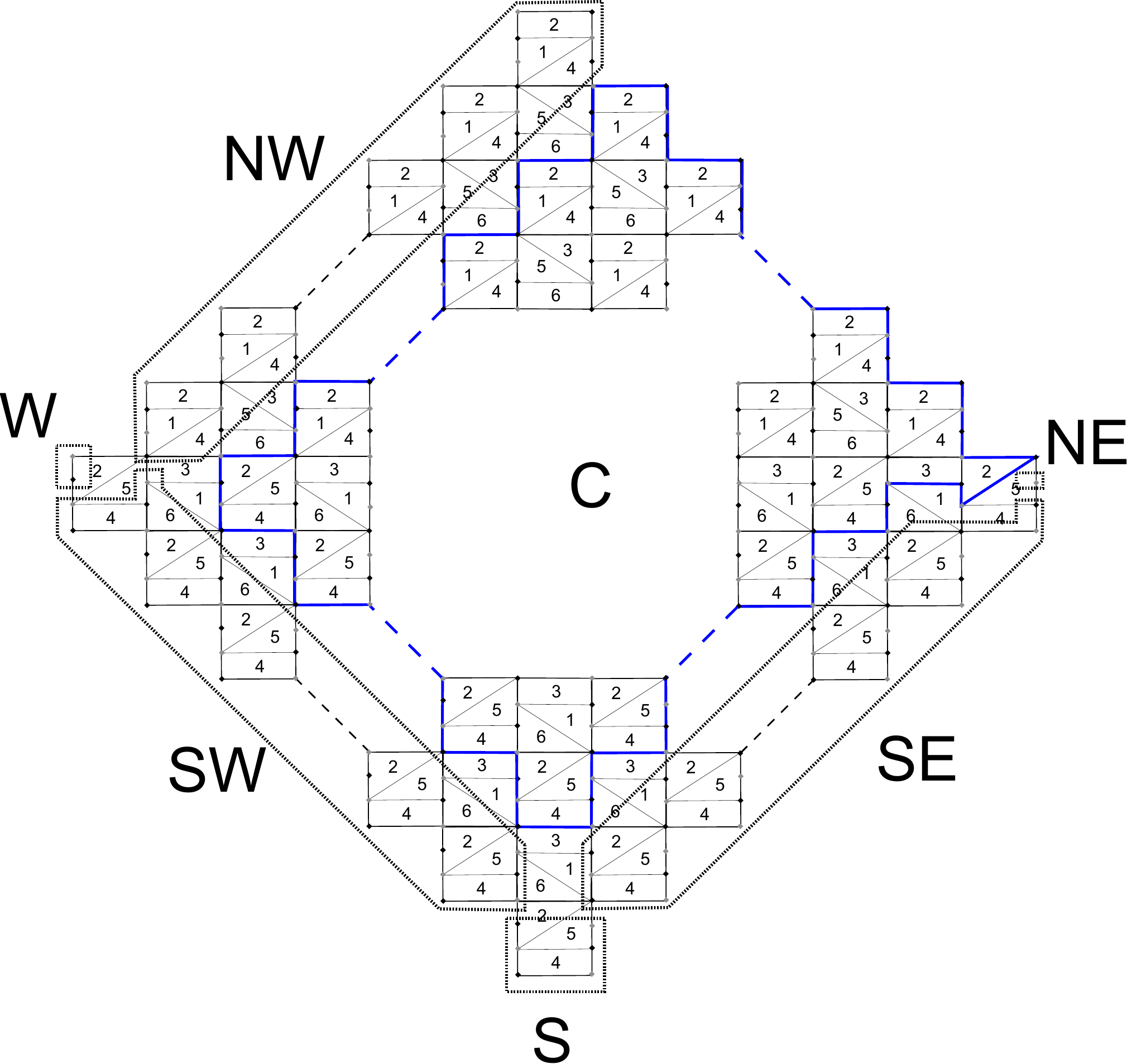}
\caption{Partitioning the vertices of $D_n$ into 7 parts}
\label{speyer_decomposing_D_n}
\end{figure}

Now for $m=n \in \mathbb{Z}_{\ge 3}$, partition the vertices of $D_n$ as shown in Figure~\ref{speyer_decomposing_D_n}. In fact $\mathbf{N}=\mathbf{E}=\emptyset$. It is easy to see that conditions 1 through 3 are all satisfied, so the method of graphical condensation can be applied, yielding \eqref{Condensation}. The LHS of \eqref{Condensation} agrees with the LHS of \eqref{w recursion 1} since $\mathbf{C}=D_{n-3/2}$. We can see from Figure~\ref{decomposing_D_n} that the RHS of \eqref{Condensation} also agrees with the RHS of \eqref{w recursion 1}:  
\begin{align*}
\mathbf{W} \cup \mathbf{NW} \cup \mathbf{SW} \cup \mathbf{C}&=D_{n-1/2},\\
\mathbf{E} \cup \mathbf{NE} \cup \mathbf{SE} \cup \mathbf{C}&=D_{n-1},\\
w(\mathbf{S}) &= \frac{1}{x_1 x_2} \frac{1}{x_3 x_4}\frac{1}{x_5 x_6};\\
\mathbf{S} \cup \mathbf{SE} \cup \mathbf{SW} \cup \mathbf{C} &= D'_{n-1/2},\\
w(\mathbf{N} \cup \mathbf{NE} \cup \mathbf{NW} \cup \mathbf{C}) &= w(D'_{n-1}) \frac{1}{x_1 x_2} \frac{1}{x_3 x_5},\\
w(\mathbf{W}) &= \frac{1}{x_2 x_3}.
\end{align*}
Note that although $\mathbf{N} \cup \mathbf{NE} \cup \mathbf{NW} \cup \mathbf{C} \neq D'_{n-1}$, its perfect matchings must contain two of the edges (colored in red), with weights $\frac{1}{x_1 x_2}$ and $\frac{1}{x_3 x_5}$, leaving a perfect matching of $D'_{n-1}$. Also, $w(\mathbf{N})=w(\mathbf{E})=w(\emptyset)=1$.

The case $n=2$ involves the diamond $D_{1/2}$ which is defined separately to be a single square labeled 2. In fact the vertices of $D_2$ can still be partitioned as shown in Figure~\ref{speyer_decomposing_D_n}.

\begin{figure}
\centering
\def\svgwidth{\columnwidth}
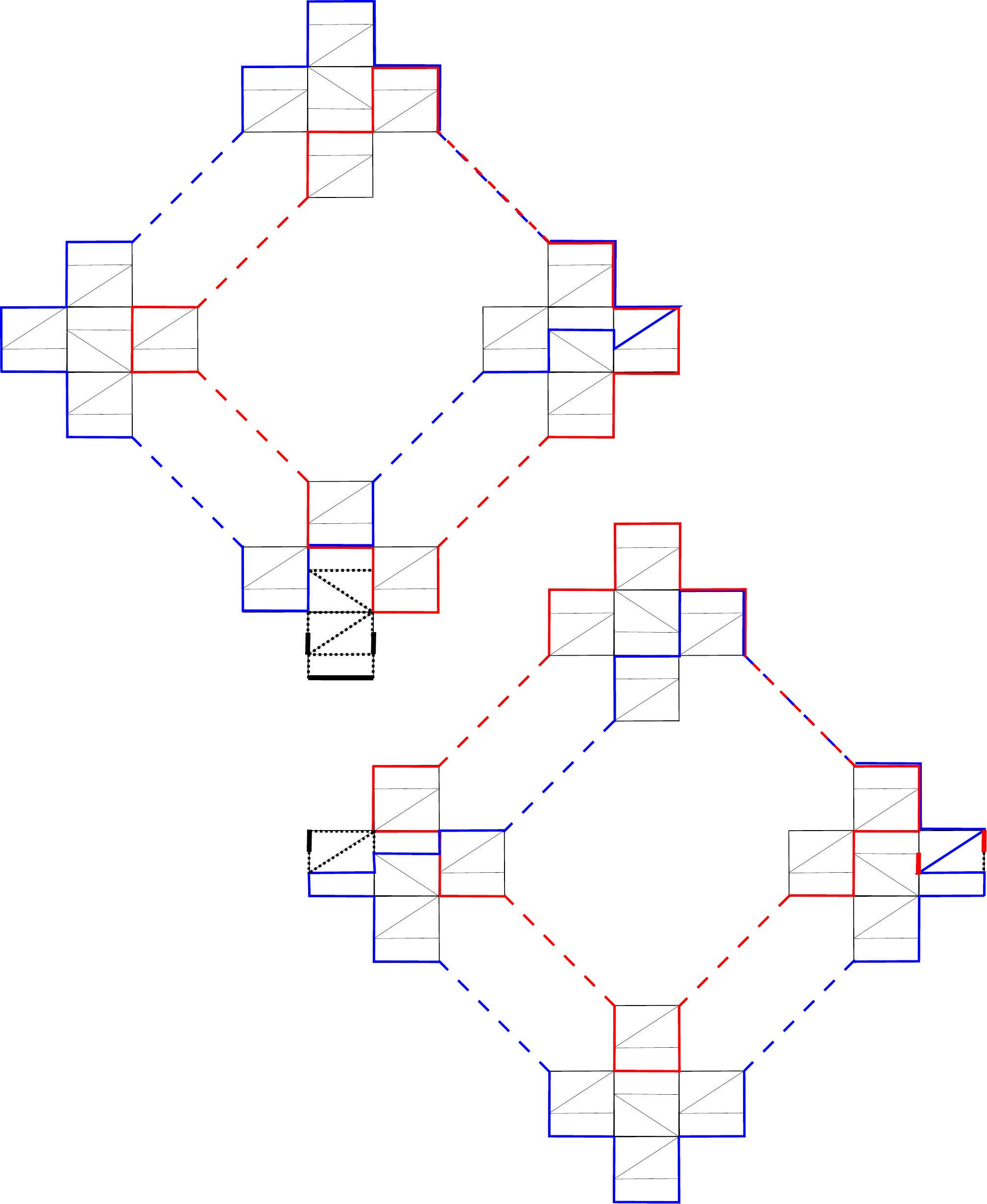
\caption{$D_n$ plus $D_{n-3/2}$ decomposed into two smaller diamonds, in two ways}
\label{decomposing_D_n}
\end{figure}

Equation \eqref{w recursion 2} is similarly proven by applying graphical condensation; see Figure~\ref{speyer_decomposing_D_nhalf} and Figure~\ref{decomposing_D_nhalf} for $n \ge 2$. In fact Figure~\ref{speyer_decomposing_D_nhalf} is valid even when $\mathbf{C}=\emptyset=D_0$, so \eqref{w recursion 2} also holds for $n=1$.

\begin{figure}
\centering
\includegraphics[width=\textwidth]{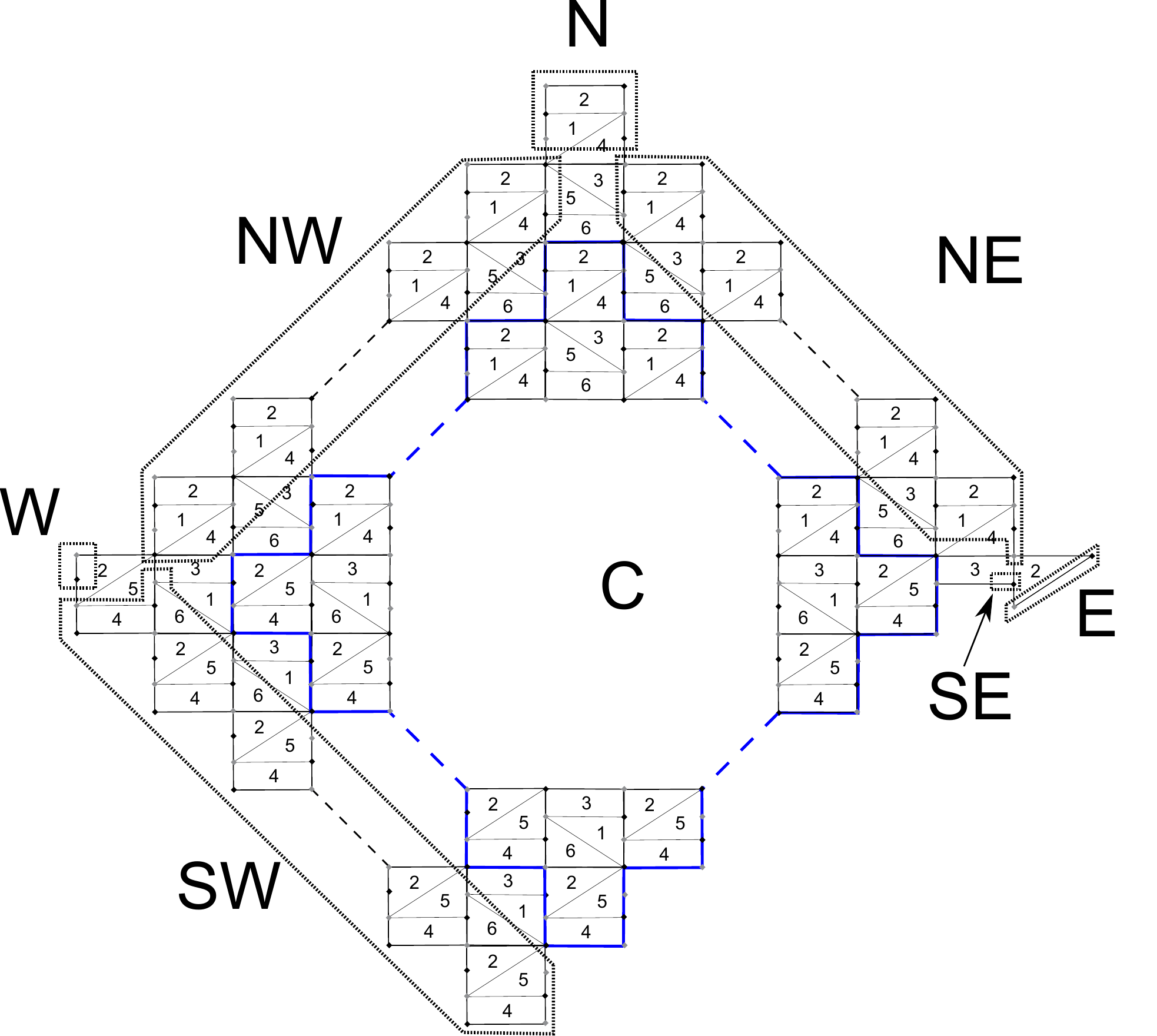}
\caption{Partitioning the vertices of $D_{n+1/2}$ into 7 parts}
\label{speyer_decomposing_D_nhalf}
\end{figure}

\begin{figure}
\centering
\def\svgwidth{\columnwidth}
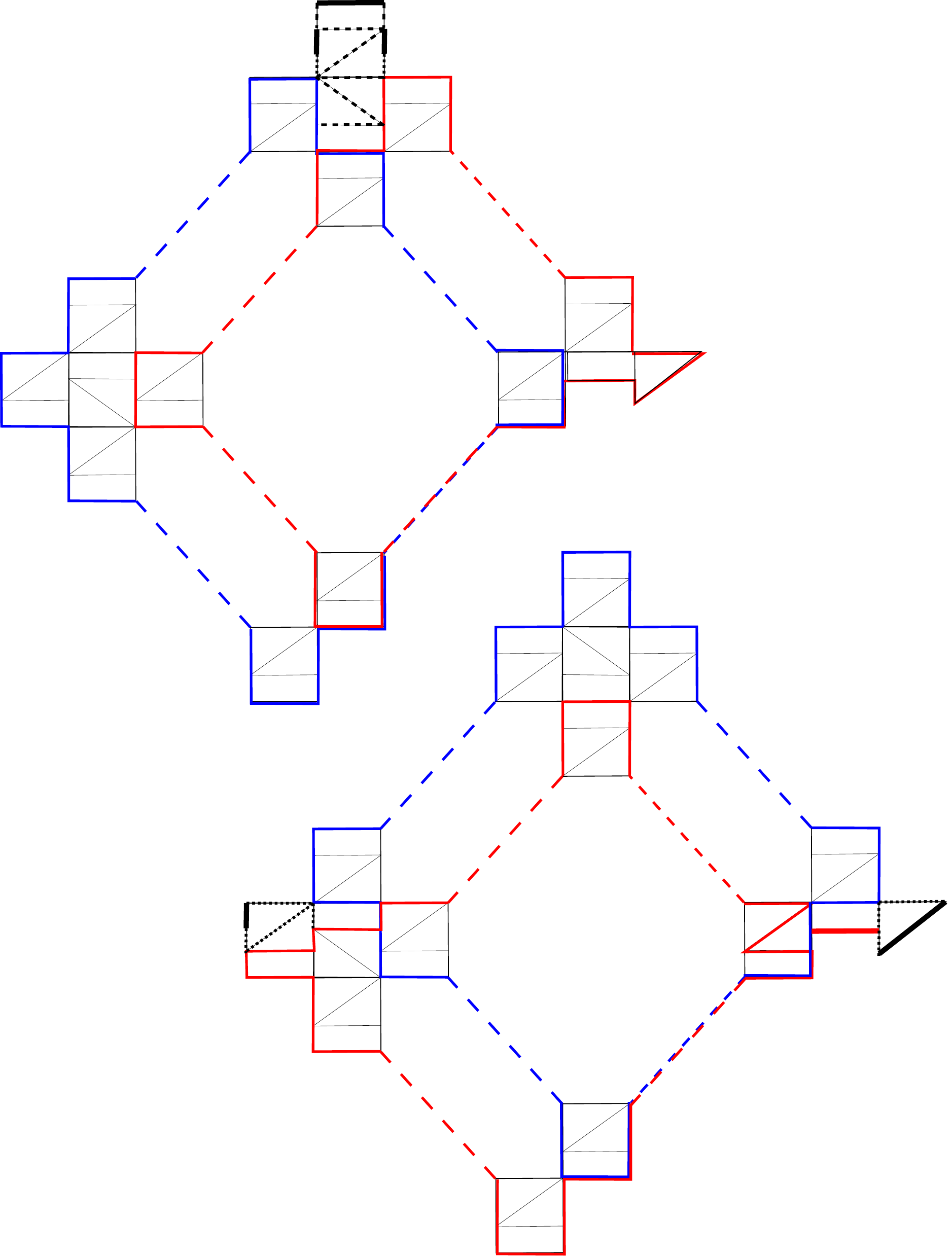
\caption{$D_{n+1/2}$ plus $D_{n-1}$ decomposed into two smaller diamonds, in two ways}
\label{decomposing_D_nhalf}
\end{figure}

\end{proof}

\section{Recursion on the covering monomials of diamonds}
\begin{prop}
For $ n \in \mathbb{Z}_{\ge 2}$,
\begin{align}\label{m recursion 1}
m(D_n)m(D_{n-3/2}) &= m(D_{n-1/2})m(D_{n-1})(x_1 x_2)(x_3 x_4)(x_5 x_6) \nonumber
\\                   &= m(D'_{n-1/2})m(D'_{n-1})(x_1 x_2)(x_2 x_3)(x_3 x_5);
\end{align}
\begin{align}\label{m recursion 2}
m(D_{n+1/2})m(D_{n-1}) &= m(D_{n})m(D_{n-1/2})(x_1 x_3)(x_2 x_6)(x_4 x_5)\nonumber
\\                       &= m(D'_{n})m(D'_{n-1/2})(x_1 x_3)(x_2 x_3)(x_2 x_5).
\end{align}
\end{prop}
\begin{proof}
We will count the squares and compute both sides. Let $n \in \mathbb{Z}_{\ge 1}$, define $ f_n \in \mathbb{N}^6 $ where $ (f_n)_i $ is the number of squares labeled $ i $ in $ D_n $.

The number of blocks of $[254]$ in $ D_n $ is 
$$ 1+ 2+ \ldots + n = \frac{1}{2} n(n+1).$$ Similarly the number of blocks of $[316]$, $[214]$, $[356]$ is respectively $ \frac{1}{2} n(n-1), \frac{1}{2} n(n-1), \frac{1}{2} (n-1)(n-2) .$

Using multi-index notation,
\begin{align*}
\mathbf{x}^{f_n} &= (x_2 x_5 x_4)^{n(n+1)/2}(x_3 x_1 x_6)^{n(n-1)/2} (x_2 x_1 x_4)^{n(n-1)/2}(x_3 x_5 x_6)^{(n-1)(n-2)/2}\\
&=\mathbf{x}^{(n(n-1), n^2, (n-1)^2, n^2, n(n-1)+1, (n-1)^2)},
\end{align*}
i.e.
$$ f_n = (n(n-1), n^2, (n-1)^2, n^2, n(n-1)+1, (n-1)^2).$$
Similarly $$ f_{n+1/2} = (n^2, n(n+1)+1, n(n-1)+1, n(n+1), n^2, n(n-1)).$$

Let $ (h_n)_i $ be the number of neighboring squares of $ D_n $ labeled $ i $, then

\begin{align*}
\mathbf{x}^{h_n} &= x_6\left[(x_3 x_6)(x_5 x_6)\right]^{n-1}\left[(x_1 x_3)(x_3 x_6)\right]^n x_3\\
			 &= \mathbf{x}^{(n, 0, 3n, 0, n-1, 3n-1)}, 
\end{align*}
and
\begin{align*}
\mathbf{x}^{h_{n+1/2}} &=x_6\left[(x_3 x_6)(x_5 x_6)\right]^n \left[(x_1 x_3)(x_3 x_6)\right]^{n-1}(x_1 x_3)(x_1 x_6 x_5)\\ 
			 &= \mathbf{x}^{(n+1, 0, 3n, 0, n+1, 3n+1)}. 
\end{align*}
So,
\begin{align}
\nonumber
m(D_n) &= \mathbf{x}^{(f_n+h_n)} \\
\nonumber
       &=\mathbf{x}^{(n(n-1), n^2, (n-1)^2, n^2, n(n-1)+1, (n-1)^2)+(n, 0, 3n, 0, n-1, 3n-1)}\\ 
\label{mDn}
       &= \mathbf{x}^{(n^2, n^2, n^2+n+1, n^2, n^2, n^2+n)}, 
\end{align}
\begin{align}
\nonumber
m(D_{n+1/2}) &= \mathbf{x}^{(f_{n+1/2}+h_{n+1/2})} \\
\nonumber
             &=\mathbf{x}^{(n^2, n(n+1)+1, n(n-1)+1, n(n+1) +(n+1, 0, 3n, 0, n+1, 3n+1)}\\
\label{mDnhalf}
             &= \mathbf{x}^{(n^2+n+1, n^2+n+1, n^2+2n+1, n^2+n, n^2+n+1, n^2+2n+1)},
\end{align}
and $m(D'_m)=\sigma (m(D_m))$ (for any positive half-integer $m$) by permuting entries in the exponent. Note that although $f$ and $h$ above give the wrong count for the special case $D_{1/2}$, \eqref{mDnhalf} for $n=0$ still gives the correct expression for $m(D_{1/2})=x_1 x_2 x_3 x_5 x_6$. Also, formally define $m(D_0)$ by \eqref{mDn}: $m(D_0)=x_3$.

So, for all $n \in \mathbb{Z}_{\ge 2}$ the expressions in \eqref{m recursion 1} all evaluate to
$$ \mathbf{x}^{(2n^2-3n+3,2n^2-3n+3,2n^2-n+2,2n^2-3n+2,2n^2-3n+3,2n^2-n+1)},$$
and for all $n \in \mathbb{Z}_{\ge 1}$ the expressions in \eqref{m recursion 2} all evaluate to
$$ \mathbf{x}^{(2 n^2-n+2, 2 n^2-n+2, 2 n^2+n+2, 2 n^2-n+1, 2 n^2-n+2, 2 n^2+n+1)}.$$

In fact, it is easier to see the proposition is true without explicit computations, but by simply looking at the ``cover" of diamonds (i.e. faces plus neighboring faces) and comparing their overlaps.

\begin{figure}
\centering
\includegraphics[scale=0.6]{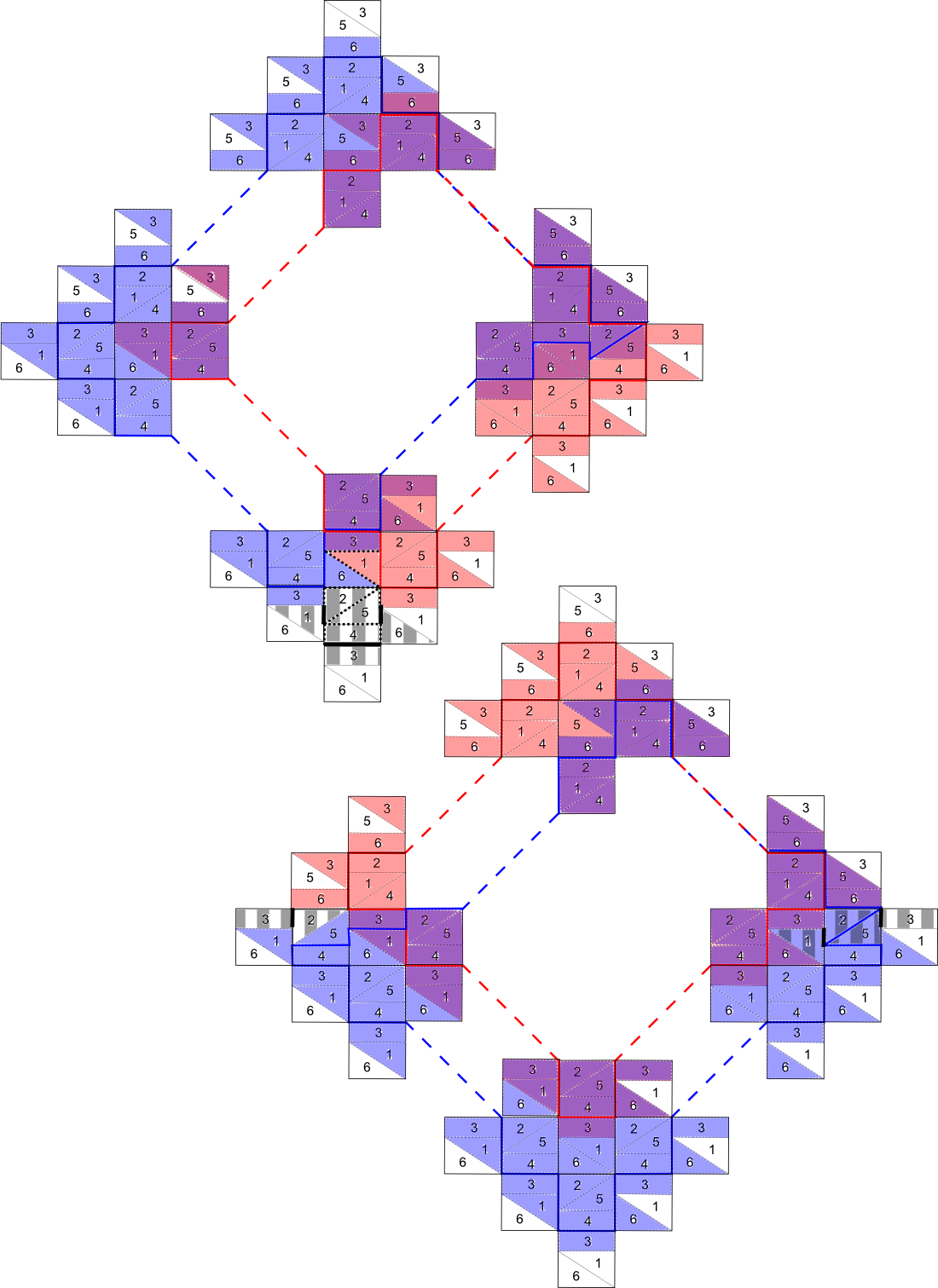}
\caption{``cover" of $D_{m}$ plus $D_{m-3/2}$ decomposed into ``cover" of two smaller diamonds, in two ways}
\label{coverD_n}
\end{figure}

In the upper half of Figure~\ref{coverD_n}, the ``covers" of $D_{n-1/2}$ (colored blue) and $D_{n-1}$ (red) intersect at a purple area, which is exactly the ``cover" of $D_{n-3/2}$, while their union together with 6 extra faces (shaded) is exactly the ``cover" of $D_n$, proving the first equality of \eqref{m recursion 1}. The lower half of Figure~\ref{coverD_n} proves the second equality of \eqref{m recursion 1}, and Figure~\ref{coverD_nhalf} proves \eqref{m recursion 2}.

Note that the 6 extra faces (shaded) in Figure~\ref{coverD_n} and Figure~\ref{coverD_nhalf} are paired up by the 3 extra edges (thickened), which combines to cancel out the extra edge weight in \eqref{w recursion 1} and \eqref{w recursion 2}.
\begin{figure}
\centering
\includegraphics[scale=0.6]{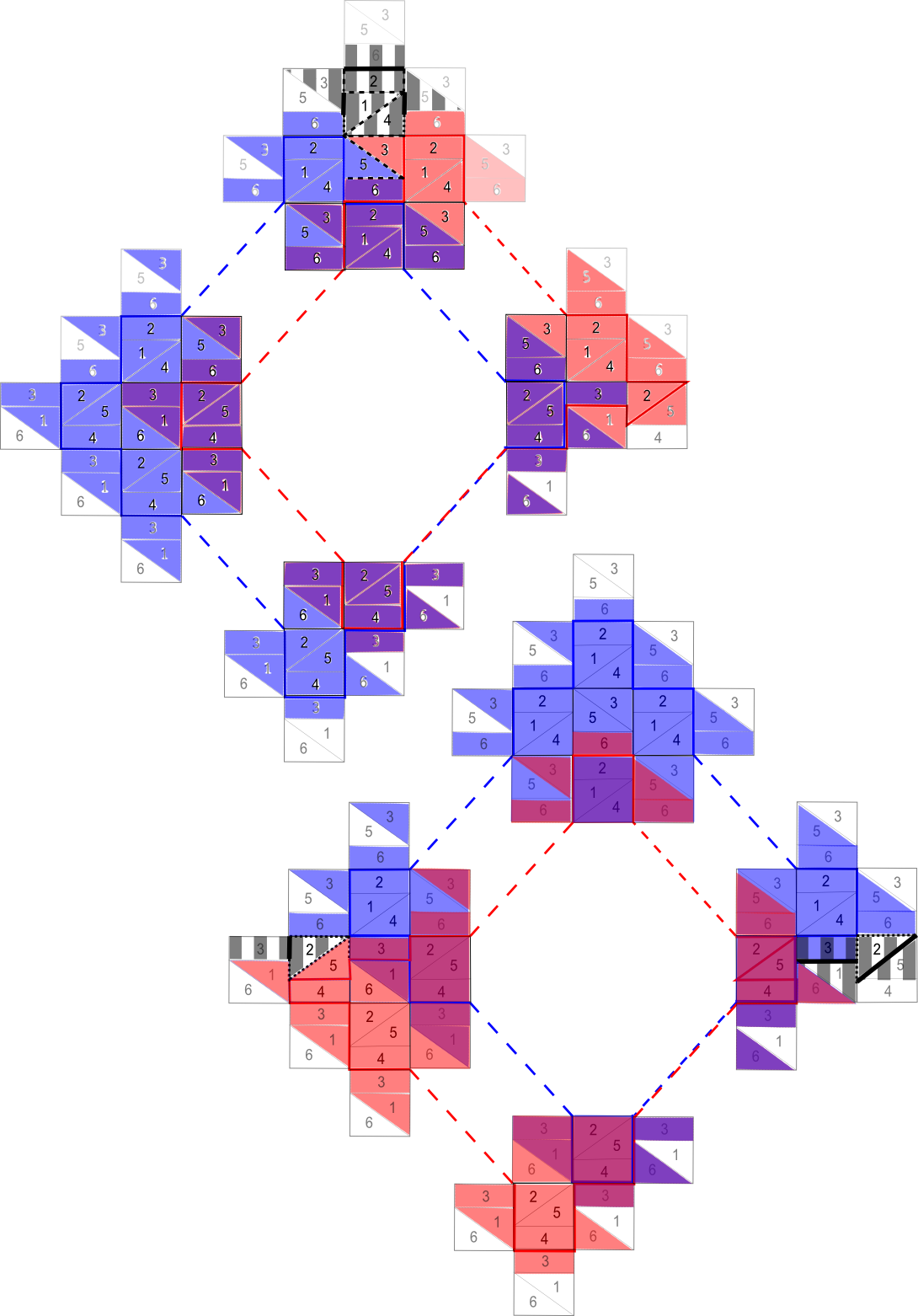}
\caption{``cover" of $D_{m+1/2}$ plus $D_{m-1}$ decomposed into ``cover" of two smaller diamonds, in two ways}
\label{coverD_nhalf}
\end{figure}

\end{proof}

\section{Proof of the theorem}
\begin{proof}
Let $N \ge 3$ and assume $y_k=w(D_{k/2})m(D_{k/2}), y'_k=w(D'_{k/2})m(D'_{k/2})$ for all $k < N$. Write $z_N=w(D_{N/2})m(D_{N/2})$. Apply \eqref{w recursion 1} and \eqref{m recursion 1} if $N$ is odd, or \eqref{w recursion 2} and \eqref{m recursion 2} if $N$ is even:
\begin{align*}
z_N y_{N-3}  &=  w(D_{N/2})m(D_{N/2})w(D_{N/2-3/2})m(D_{N/2-3/2})\\
             &=w(D_{N/2-1/2})w(D_{N/2-1})m(D_{N/2-1/2})m(D_{N/2-1})\\
             &\;+w(D'_{N/2-1/2})w(D'_{N/2-1})m(D'_{N/2-1/2})m(D'_{N/2-1})\\
             &= y_{N-1} y_{N-2} + y'_{N-1} y'_{N-2}\\
		  &= y_N y_{N-3}
\end{align*}
agrees with  exchange relation \eqref{exchange}, so $z_N=w(D_{N/2})m(D_{N/2}) = y_N$. Act by $\sigma$ on both sides: $w(D'_{N/2})m(D'_{N/2}) = y'_N$. The base case of $N=0$ becomes: $y_0=1 \cdot x_3=w(D_0)m(D_0)$, $y'_0=1 \cdot x_6=w(D'_0)m(D'_0)$ (not interesting, but it is formally consistent). The base cases of $N=1, 2$ can be checked directly, or are tree phenomena as discussed by Jeong \cite{J11}.
\end{proof} 

\section{Acknowledgments}
This research was conducted in the 2012 summer REU program in the University of Minnesota, Twin Cities. I was supported by the Columbia University Rabi Scholars Program. I would like to thank Vic Reiner, Gregg Musiker, Pavlo Pylyavskyy and Dennis Stanton for mentoring the program. I am especially grateful to Gregg Musiker for his guidance and encouragement throughout this project. I also thank Gregg Musiker and Alexander Garver for reading through and commenting on an earlier draft. This REU report previously appeared on Vic Reiner's REU website at 

{\tt \verb+http://www.math.umn.edu/~reiner/REU/Zhang2012.pdf+}.

\bibliographystyle{plain}
\bibliography{mathREU}
\end{document}